\documentclass[10pt]{article}
\usepackage{amsmath,amssymb,amsthm,graphicx,url,colordvi,enumerate,amsfonts,graphics,mathrsfs,bbm,yfonts}
\usepackage[left=1.5in,top=1.25in,right=1.5in,nohead,bottom=1.25in]{geometry}

\DeclareFontFamily{OT1}{pzc}{}
\DeclareFontShape{OT1}{pzc}{m}{it}{<-> s * [1.150] pzcmi7t}{}
\DeclareMathAlphabet{\mathpzc}{OT1}{pzc}{m}{it}

\newcommand{\N}{\mathbb{N}}

\newcommand{\C}{\mathbb{C}}
\newcommand{\bA}{\mathbf{A}}

\newcommand{\inv}{^{-1}}
\newcommand{\Sph}{{\mathscr T}}
\newcommand{\T}{{\mathscr T}}
\newcommand{\Susp}{{\cal S}}
\newcommand{\I}{{\cal I}}
\newcommand{\F}{{\cal F}}

\newcommand{\E}{{\cal E}}
\newcommand{\M}{{\cal M}}

\DeclareMathOperator{\Ho}{Ho}
\newcommand{\Top}{\mathpzc{Top}}
\newcommand{\hTop}{\Ho(\mathpzc{Top})}

\newcommand{\rk}{{\textfrak{r}}}

\newcommand{\D}{{\cal D}}

\newtheorem{theorem}{Theorem}[section]
\newtheorem{lemma}[theorem]{Lemma}
\newtheorem{conjecture}[theorem]{Conjecture}
\newtheorem{proposition}[theorem]{Proposition}

\newtheorem{definition}[theorem]{Definition}
\newtheorem{corollary}[theorem]{Corollary}

\newtheorem{question}[theorem]{Question}

\theoremstyle{remark}
\newtheorem{remark}[theorem]{\bf Remark}
\newtheorem{example}[theorem]{\bf Example}
\newtheorem*{acknowledgements}{Acknowledgements}

\DeclareMathOperator{\colim}{colim}
\DeclareMathOperator{\hocolim}{hocolim}
\DeclareMathOperator{\lat}{{\mathcal L}}
\DeclareMathOperator{\cl}{cl}

\DeclareMathOperator{\OA}{\AA}

\begin{document}
\title{Topological representations of matroid maps}

\author{Matthew T. Stamps \\ Aalto University \\ \tt matthew.stamps@aalto.fi}

\date{May 1, 2012}

\maketitle

\begin{abstract}
The Topological Representation Theorem for (oriented) matroids states that every (oriented) matroid arises from the intersection lattice of an arrangement of codimension one homotopy spheres on a homotopy sphere.  In this paper, we use a construction of Engstr\"{o}m to show that structure-preserving maps between matroids induce topological mappings between their representations; a result previously known only in the oriented case.  Specifically, we show that weak maps induce continuous maps and that this process is a functor from the category of matroids with weak maps to the homotopy category of topological spaces.  We also give a new and conceptual proof of a result regarding the Whitney numbers of the first kind of a matroid.  
\end{abstract}

\section{Introduction}

A celebrated achievement in the theory of oriented matroids is the \emph{Topological Representation Theorem} of Folkman and Lawrence \cite{Bjorner95,BLSWZ} which states that every oriented matroid arises from an arrangement of codimension one pseudospheres on a sphere.  This is converse to the more straightforward observation that the intersection lattice of an arrangement of codimension one homotopy spheres on a homotopy sphere is always isomorphic to the lattice of flats of a matroid.  The original proof of the Topological Representation Theorem relied heavily on the orientation data, so it was long believed that an analogous representation theorem for unoriented matroids did not exist.  In 2003, however, Swartz \cite{Swartz03} gave a surprising proof that \emph{any} matroid can, indeed, be represented by a codimension one homotopy sphere arrangement.  Anderson recently gave an explicit construction for such arrangements in \cite{Anderson2010}.  Shortly after, Engstr\"{o}m \cite{Engstrom10}  showed that for every matroid $M$, there is infinite family of subspace arrangements, indexed by finite regular CW complexes, called $X$-arrangements, whose intersection lattices are all isomorphic to the lattice of flats of $M$.
These representations can be constructed explicitly using diagrams of spaces, are equipped with natural geometric realizations, and include the codimension one spherical arrangements described above.

Accompanying the Topological Representation Theorem for oriented matroids is a result of Anderson \cite{Anderson2001}, which states that weak maps between oriented matroids induce topological mappings between their representations.  In this paper, we prove that weak maps between matroids induce continuous maps between their Engstr\"om representations and show that several combinatorial properties of weak maps have nice topological interpretations within this framework.  For example, the continuous representation of a weak map is equivariant whenever $X$, in the representing $X$-arrangements, supports a free action of a group.  We also apply our techniques to give a new and conceptual proof of a result regarding the Whitney numbers of the first kind of a matroid, see Proposition \ref{whit_dec}.

Following the spirit of \cite{Engstrom10}, much of the work in this paper involves setting the correct viewpoint under which the desired results fall out naturally from the combinatorial comparison lemmas established in \cite{WZZ99} and \cite{ZZ93}.  A technical aspect, which requires careful bookkeeping, is that Engstr\"{o}m's construction involves a choice of poset map from the lattice of flats of a matroid to a boolean lattice.  The choice of such a map allows one to construct different geometric realizations for a given representation, which is vital in a number of applications.  For this reason, we introduce the notions of immersed matroids and admissible weak maps.  We prove our results first for these new objects and then show that every weak map between matroids can be written as an admissible weak map between immersed matroids.  

Here is an overview of our approach: First, we recall that a weak map $\tau$ between matroids $M$ and $N$ induces a poset map $\tau^{\#}$ between their lattices of flats, $\lat(M)$ and $\lat(N)$.  Then, we construct Engstr\"{o}m representations on $\lat(M)$ and $\lat(N)$, using diagrams of spaces, and show that $\tau^{\#}$ induces a morphism of diagrams.  A result in \cite{WZZ99} yields the desired continuous maps between representations.
Next, we introduce the category $\OA(r,n)$ of rank $r$ immersed matroids on $n$ elements with admissible weak maps and show that the above process yields a functor from $\OA(r,n)$ to $\hTop$, the homotopy category of $\Top$. This extends to a functor from $\M(r,n)$, the category of rank $r$ matroids on $n$ elements with weak maps, to $\hTop$ by composing with a natural map $\M(r,n) \to \OA(r,n)$.  

Throughout the paper, a \emph{representation} always refers to a topological space and a \emph{geometric realization} refers to its classical definition from combinatorial topology.  These should not be confused with the alternative notions of representability or realizability (over some field) from matroid theory.

The paper is organized as follows:  In Section 2, we introduce the basic definitions and tools for both matroid theory and diagrams of spaces, followed by a review of Engstr\"{o}m's Representation Theorem for Matroids.  In Section 3, we state and prove our main results; first, that (admissible) weak maps induce continuous maps between the topological representations of (immersed) matroids; second, that the Whitney numbers of the first kind are encoded in the Betti numbers of certain topological representations of a matroid; and third, that the continuous representations of surjective weak maps give surjections in homology.  We show that Engstr\"{o}m's construction is a functor, up to homotopy, with respect to (admissible) weak maps in Section 4 and we conclude the paper with some open questions and future work in Section 5.  

\section{Background}

In this section, we give a brief introduction to the theory of matroids along with a tool box of the necessary topological machinery for the paper and a review of Engstr\"{o}m's construction for representing matroids with subspace arrangements.    

\subsection{Matroids}   

We begin with a review of the key definitions and theorems about matroids and geometric lattices.  For a more in-depth introduction to matroid theory, see \cite{Oxley} and \cite{White}.

\begin{definition}
A \emph{matroid} $M$ is a pair $(E,\I)$ where $E$ is a finite set and $\I \subseteq 2^E$ satisfying  \begin{itemize} \item $\emptyset \in \I$; \item if $Y \in \I$ and $X \subseteq Y$, then $X \in \I$; \item if $X , Y \in \I$ with $ |X| < |Y|$, there exists $y \in Y \setminus X$ such that $X \cup \{y\} \in \I$. \end{itemize}  The set $E = E(M)$ is called the \emph{ground set} of $M$ and $\I = \I(M)$ is called the collection of \emph{independent sets} of $M$.   \end{definition} 

The \emph{rank function} $\rk : 2^E \to \N$ is given by $\rk(X) = \max \{|Y| : Y \subseteq X \ \& \ Y \in \I\}$.  We define the \emph{rank} of $M$, denoted $\rk(M)$, by $\rk(E)$.  When it is necessary to distinguish the rank function of a matroid $M$ from that of other matroids, we denote its rank function by $\rk_M$.  The \emph{closure} of a set $X \subseteq E$ is $$\cl(X) = \{x \in E \ | \ \rk(X \cup \{x\}) = \rk(X)\}.$$ A \emph{flat} of $M$ is any set $X \subseteq E$ where $X = \cl(X)$.  Matroids have many different equivalent definitions, one of which is in terms of lattices.  In particular, the flats of a rank $r$ matroid form a graded, rank $r$, \emph{geometric} lattice where meets are given by intersections and the join of all atoms is $E$, see \cite{Bjorner82} and \cite{Stanley}.

\begin{definition} A lattice $L$ is called \emph{geometric} if (1) it is semimodular (i.e. $$\rk(p) + \rk(q) \geq \rk(p \wedge q) + \rk(p \vee q)$$ for all $p,q \in L$); and (2) every element is the join of atoms.
\end{definition}

We denote the lattice of flats of a matroid $M$ by $\lat(M)$.  If $M$ is simple (i.e. $cl(X) = X$ for all $X \subset E(M)$ with $|X| \leq 1$), then $M$ is completely determined by $\lat(M)$.

\begin{theorem}[\cite{Oxley}, Theorem 1.7.5] A lattice $L$ is geometric if and only if $L \cong \lat(M)$ for some matroid $M$.  If one restricts to simple matroids, then $M$ is unique.
\end{theorem}

We will use the lattice-theoretic interpretation of matroids throughout the rest of the paper.  The following two lemmas will be especially useful:

\begin{lemma}[\cite{Stanley}, Corollary 3.9.3 (dual version)]\label{enum_lem}
Let $L$ be a geometric lattice and let $p$ be any rank one element of $L$.  Then the M\"{o}bius function $\mu_{L}$ of $L$ satisfies $$\mu_{L}(\hat{0},\hat{1}) = - \sum\limits_{\begin{matrix}\scriptsize \text{$q$ coatom of $L$} \\ \scriptsize \text{$q \ngeq p$}\end{matrix}} \mu_{L}(\hat{0},q).$$
\end{lemma}

\begin{lemma}[\cite{Rota64}] \label{pos_lem}
If $L$ is a finite geometric lattice, then $\mu_L(\hat{0},p) \neq 0$ for each $p \in L$.
\end{lemma}

Next, we introduce the structure-preserving maps between matroids.  There are two common notions for a morphism between matroids --- weak and strong maps.  As matroids are combinatorial analogs of vector configurations, matroid maps are combinatorial analogs of linear transformations, in the sense that they map vectors to vectors and dependent sets to dependent sets.  When working with matroid maps, it is customary to adjoin a zero element to each matroid $M = (E,\I)$ to get the matroid $M_o = (E \cup o, \I)$, see \cite{White}.  This allows us to express the deletion of an element $e$ as a set map (in which $e \mapsto o$).  The alteration has no effect on our work since $\lat(M)$ and $\lat(M_o)$ are always isomorphic as lattices, but the notation can be rather cumbersome.  Thus, whenever we consider a function $f : M \to N$ between matroids in this paper, we mean a set map from $E(M) \cup o$ to $E(N) \cup o$ mapping $o$ to $o$.

\begin{definition}
A \emph{weak map} is a function $\tau : M \to N$ satisfying the condition if $X \subseteq E(M)$ such that $\tau|_X$ is injective and $\tau(X) \in \I(N)$, then $X \in \I(M)$. 
\end{definition} 

An immediate consequence of this definition, see (\cite{White}, Section 9.1), is the following useful characterization of weak maps:  A function $\tau$ is a weak map from $M$ to $N$ if and only if for all $X \subseteq E(M)$, $$\rk_{N}(\tau(X)) \leq \rk_M(X).$$

\begin{definition}
A \emph{strong map} is a function $\sigma : M \to N$ satisfying the condition that the preimage of any flat in $N$ is a flat of $M$. 
\end{definition}

Just as with matroids, weak and strong maps have nice lattice-theoretic interpretations.  Any set map $f : E(M) \to E(N)$ induces an order-preserving map $f^{\#} : \lat(M) \to \lat(N)$ given by $f^{\#}(X) = cl(f(X))$ for each $X \subseteq E(M)$.  By the characterization above, a set map $\tau : E(M) \cup o \to E(N) \cup o$ mapping $o$ to $o$ is a weak map from $M$ to $N$ if and only if $\tau^{\#} : \lat(M) \to \lat(N)$ is a poset map where $\rk_{N}(\tau^{\#}(X)) \leq \rk_M(X)$ for all $X \in \lat(M)$.  In addition, a weak map $\sigma$ is strong if we add the requirement that $\sigma^{\#}$ is join-preserving (i.e. $\sigma^{\#}(X \vee Y) = \sigma^{\#}(X) \vee \sigma^{\#}(Y)$ for all $X,Y \in \lat(M))$.  A weak map $\tau : M \to N$ is called \emph{non-annihilating} if and only if $\tau^{\#}$ maps atoms to atoms.

\begin{remark}\label{surj_weak_map} If $\tau : M \to N$ is a surjective weak map, then each $Y \in E(N)$ can be rewritten as $Y' = \tau\inv(Y)$ to form a matroid $N' \cong N$ with ground set $E(M)$ such that $\tau$ is equivalent to $id : M \to N'$.  So, it is often sufficient to restrict our attention to identity maps between matroids on the same ground set.
\end{remark}

The next two lemmas will be useful for the applications in Section \ref{Apps}.

\begin{lemma} \label{surj_poset}
If $\tau : M \to N$ be a surjective weak map, then the poset map $\tau^{\#} : \lat(M) \to \lat(N)$ is surjective and for every $Y \in \lat(N)$, there exists $X \in {(\tau^{\#})}\inv(Y)$ with $\rk_M(X) = \rk_N(Y)$.  \end{lemma}

\begin{proof}
By Remark \ref{surj_weak_map}, we may assume that $\tau$ is the identity map between matroids with a common ground set $E$ and hence that $\tau^\#(X) = cl_N(X)$ for all $X \in \lat(M)$.  Let $Y \in \lat(N)$ and $Y_I \in \I(N)$ be a maximal independent set of $N$ contained in $Y$.  Since $\tau$ is a weak map, $Y_I \in \I(M)$ as well.  So, consider $X = cl_M(Y_I) \in \lat(M)$.  By definition, $\rk_M(X) = |Y_I| = \rk_N(Y)$ and since $Y_I \subseteq X$, we also get $Y = cl_N(Y_I) \subseteq cl_N(X)$.  However, since $\tau$ is a weak map, we know that $\rk_N(\tau^{\#}(X)) = \rk_N(cl_N(X)) \leq \rk_M(X)$ and thus, $\tau^{\#}(X) = cl_N(X) = Y$.  
\end{proof}

\begin{definition}
For a rank $r$ matroid $M$, the $k$th \emph{truncation} $T^k(M)$ of $M$ is the matroid of rank $r - k$ on $E(M)$ whose rank function $\rk$ is given as follows:  For $X \subseteq E(M)$, $$\rk(X) = \begin{cases} r - k & \rk_M(X) \geq r - k, \\ \rk_M(X) & \text{otherwise.} \end{cases}$$ 
\end{definition} 

It is simple to check that the identity map on $E(M)$ gives a weak map $id_k : M \to T^k(M)$ for each $k \in [r]$ and that every surjective weak map $\tau : M \to N$ induces a weak map $\tau_k : T^k(M) \to N$ for each $k \in [r-\rk(N)]$.  In fact, one can show the following:  

\begin{lemma}[\cite{White}, Lemma 9.3.1]\label{weak_fact}
If $\tau : M \to N$ is a surjective weak map and $k = \rk(M)-\rk(N)$, then $\tau = \tau_k \circ id_k$.  In other words, every surjective weak map from $M$ to $N$ factors uniquely through $T^{\rk(M)-\rk(N)}(M)$.
\end{lemma}

We conclude this section with a common statistic on matroids which will give a nice combinatorial formula for the homotopy type of a topological representation.

\begin{definition}
The \emph{Whitney numbers} $w_k(M)$ \emph{of the first kind} are given by $$w_k(M) = \sum\limits_{X \ : \ \rk(X) = k} |\mu(\hat{0},X)|,$$ where the sum is over all the rank $k$ flats $X$ in $\lat(M)$, and $\mu$ is the M\"{o}bius function of $\lat(M)$.
\end{definition}

When we restrict our attention to matroids of a fixed rank, the Whitney numbers of the first kind behave very predictably with respect to weak maps \cite{Lucas74,Lucas75}.

\begin{proposition}[\cite{White}, Corollary 9.3.7]\label{whit_dec}
If $\tau : M \to N$ is a surjective weak map and $\rk(M) = \rk(N) = r$, then $w_k(M) \geq w_k(N)$ for each $k \in \{0,...,r\}$.
\end{proposition}

In the case that $\tau$ in Proposition \ref{whit_dec} is a strong map, we get that $w_k(M) = w_k(N)$ for each $k \in \{0,...,r\}$ which is evidenced in the following proposition.

\begin{proposition}[\cite{Oxley}, Corollary 7.3.4]\label{strong_cor}
If $\sigma : M \to N$ is a surjective strong map and $\rk(M) = \rk(N)$, then $M \cong N$.
\end{proposition}

\subsection{Diagrams of Spaces} \label{DS}

Next, we build up the necessary topological machinery in the way of diagrams of spaces.  A \emph{diagram of spaces} is a (covariant) functor $\D : I \to \Top$ for some small category $I$.   In this paper, we are only concerned with the case where $I$ is a finite poset. Recall that, in the language of category theory, a poset $P = (S,\leq)$ over a set $S$ is a category whose objects are the elements of $S$ and between any two points $p,q \in S$, there is a unique morphism $p \to q$ when $p \geq q$ or no morphism otherwise.  In this setting, a functor $\D : P \to \Top$ consists of the following data: 

\begin{definition}
A $P$-diagram of spaces $\D$ consists of \begin{itemize} \item a finite poset $P$,  \item a CW complex $D(p)$ for every $p \in P$, \item a continuous map $d_{pq} : D(p) \to D(q)$ for every pair $p \geq q$ of $P$ satisfying \\ $d_{qr}\circ d_{pq}(x) = d_{pr}(x)$ for every triple $p \geq q \geq r$ of $P$ and $x \in D(p)$. \end{itemize}
\end{definition}

To every diagram $\D$, one can associate a topological space via a (homotopy) colimit.

\begin{definition}
The \emph{colimit} of a diagram $\D : P \to \Top$ is the space $$\colim_P \D = \coprod\limits_{p \in P} D(p) \ /  \sim$$ where the relation $\sim$ is generated by $x \sim y$ for each $x \in D(p)$ and $y \in D(q)$ if and only if $d_{pq}(x) = y$.
\end{definition} 

Colimits of diagrams of spaces appear all over in combinatorics and discrete geometry and are very convenient to work with; however, they have a significant drawback, in the sense that they do not behave well with respect to homotopy equivalences, see Appendix A for more details.  For this reason, we use the following object which is slightly more complicated, but much more compatible with homotopy theory.  

\begin{definition}
The \emph{homotopy colimit} of a diagram $\D : P \to \Top$ is the space $$\hocolim_P \D = \coprod\limits_{p \in P} (\Delta(P_{\leq p}) \times D(p)) \ /  \sim$$ where $\sim$ is the transitive closure of the relation $(a,x) \sim (b,y)$ for each $a \in \Delta(P_{\leq p})$, $b \in \Delta(P_{\leq q})$, $x \in D(p)$ and $y \in D(q)$ if and only if $p \geq q$, $d_{pq}(x) = y$, and $a = b$.
\end{definition}

Whenever the spaces in a $P$-diagram $\D$ are metrizable (which is true for every space considered in this paper), the homotopy colimit of $\D$ admits an explicit geometric realization.

\begin{definition}\label{geometric}
For any $P$-diagram $\D$, define ${\cal J}(\D)$ as the join of all spaces in the diagram realized by embedding them in skew affine subspaces.  The points of ${\cal J}(\D)$ can be parametrized as $$\left \{ \sum\limits_{p \in P} t_p x_p \ \right | \left. \ x_p \in D(p), \ 0 \leq t_p \leq 1 \text{ for all } p \in P, \text{ and } \sum\limits_{p \in P} t_p = 1 \right \}.$$  Then the \emph{geometric realization} of $\hocolim_P \D$ consists of the following set of points: $$\{t_0x_0 + ... + t_mx_m \in {\cal J}(\D) \ | \ x_i \in D(p_i),  p_0 \leq ... \leq p_m,  d_{p_{i+1} p_i}(x_{i+1}) = x_i\}.$$
\end{definition}

The benefit of using homotopy colimits, rather than ordinary colimits, is revealed in the following lemma:

\begin{lemma}[\bf Homotopy Lemma, \cite{ZZ93}]\label{HL}
Let $\D$ and $\E$ be $P$-diagrams with homotopy equivalences $\alpha_p : D(p) \to E(p)$ for every $p \in P$ such that the following diagram commutes for all $p > q$.$$\begin{matrix} & & \alpha_p & & \\ & D(p) & \longrightarrow & E(p) & \\ d_{pq} & \downarrow & & \downarrow & e_{pq} \\ & D(q) & \longrightarrow & E(q) & \\ & & \alpha_q & & \end{matrix}$$  Then $\{\alpha_p\}_{p \in P}$ induces a homotopy equivalence $\alpha : \hocolim_P \D \to \hocolim_P \E.$
\end{lemma}

Next, we consider the notion of a structure-preserving map between diagrams of spaces.    

\begin{definition}
A morphism $(f,\alpha) : \D \to \E$ of diagrams $\D : I \to \Top$ to $\E : J \to \Top$ is a functor $f : I \to J$ together with a natural transformation $\alpha$ from $\D$ to $\E \circ f$. 
\end{definition}

Our main interest in morphisms of diagrams is that they induce continuous maps between the corresponding homotopy colimits.  The next two results are lifted from \cite{WZZ99}, but are commonly known in the theory of homotopy colimits, see \cite{HV}, \cite{Segal}, and \cite{Vogt}.  For the remainder of this section, let $\D : P \to \Top$ and $\E : Q \to \Top$ be diagrams of spaces over posets $P$ and $Q$.

\begin{proposition} \label{morphdiag}
If $(f,\alpha) : \D \to \E$ is a morphism of diagrams, then $(f,\alpha)$ induces a continuous map $f^* : \hocolim_P \D \to \hocolim_Q \E$.
\end{proposition}

In fact, the map $f^*$ is completely explicit:  If one writes each point in $\Delta(P_{\leq p}) \times D(p)$ as $(\lambda_1 p_1 + \cdots + \lambda_k p_k , x)$ where $p_1 \leq p_2 \leq \cdots \leq p_k = p$, $\lambda_i \geq 0$, $\sum_i \lambda_i = 1$, and $x \in D(p)$, then $$f^*(\lambda_1 p_1 + \cdots + \lambda_k p_k , x) = (\lambda_1 f(p_1) + \cdots + \lambda_k f(p_k) , \alpha_p(x)).$$  There is a simple criterion for when two morphisms of diagrams induce homotopic maps between the corresponding homotopy colimits.

\begin{proposition} \label{morph_hom}
If $(f,\alpha),(g,\beta): \D \to \E$ are morphisms of diagrams such that there exists a natural transformation $\gamma : f \to g$ satisfying $E(\gamma) \circ \alpha = \beta$, then $f^* \simeq g^*$.
\end{proposition}

\begin{remark}\label{morph_hom_rem} Whenever the components of $\alpha$ and $\beta$ are inclusions and the maps in $\D$ and $\E$ are also inclusions, any pair $f,g : P \to Q$ where $f(p) \geq g(p)$ for every $p \in P$ satisfies the conditions of Proposition \ref{morph_hom} and thus, $f^* \simeq g^*$.    
\end{remark}

In order to compare homotopy colimits of diagrams over different posets, we will use the following variation of the Upper Fiber Lemma \cite{WZZ99}, which is a generalization of Quillen's Fiber lemma \cite{Bjorner95}. First, we establish a ``gluing lemma'':

\begin{lemma}\label{GL}
Let $X$ and $Y$ be CW complexes and let $A,B \subseteq X$ and $C,D \subseteq Y$ be subcomplexes such that $X = A \cup B$ and $Y = C \cup D$.  If $f : X \to Y$ is a continuous map satisfying \begin{itemize}\item $f(A) \subseteq C$ and $f(B) \subseteq D$, \item $f|_A$, $f|_B$ induce surjections, $\tilde{f_A} : H_*(A) \twoheadrightarrow H_*(C)$, $\tilde{f_B}:H_*(B) \twoheadrightarrow H_*(D)$, \end{itemize} then $f$ induces a surjection, $\tilde{f}: H_*(X) \twoheadrightarrow H_*(Y)$.
\end{lemma}

\begin{proof}
Consider the induced chain map on the Mayer-Vietoris sequences of $X = A \cup B$ and $Y = C \cup D$.  We get the following commutative diagram: $$\begin{matrix} & H_*(A) \oplus H_*(B) & \overset{\delta}{\longrightarrow} & H_*(X) & \\ {\tiny (\tilde{f_A},\tilde{f_B})} & \downarrow \ \ \ \ \ \ \ \ \ \ \downarrow & & \downarrow & \tilde{f} \\ & H_*(C) \oplus H_*(D) & \underset{\delta}{\longrightarrow} & H_*(Y) & \end{matrix}$$ where $\delta$ is the standard difference map.  Since $\tilde{f_A}$, $\tilde{f_B}$, and $\delta$ are surjections, $\delta \circ (\tilde{f_A},\tilde{f_B}) = \tilde{f} \circ \delta$ is a surjection and hence, $\tilde{f}$ is a surjection too.  
\end{proof}

\begin{lemma}\label{VUFL}
Let $(f,\alpha) : \D \to \E$ be a morphism of diagrams.  If $f^*$ induces a surjection $$H_*(\hocolim_{\{p \in P : f(p) \geq q\}} \D) \twoheadrightarrow H_*(\hocolim_{Q_{\geq q}} \E)$$ for each $q \in Q$, then $f^*$ induces a surjection $H_*(\hocolim_P \D) \twoheadrightarrow H_*(\hocolim_Q \E)$.
\end{lemma}

\begin{proof}
This proof follows that of Theorem 3.8 in \cite{WZZ99}.  We induct on $|Q|$.  If either $|Q| = 1$ or there is a unique minimal element in $Q$, the statement is trivial.  So, suppose that $q \in Q$ is one of several minimal elements $Q$.  Then write $Q = Q_{\geq q} \cup Q \setminus \{q\}$ and $P = f\inv(Q_{\geq q}) \cup f\inv(Q \setminus \{q\})$.  By induction, $f^*$ induces surjections $H_*(\hocolim_{f\inv(Q_{\geq q})} \D) \twoheadrightarrow H_*(\hocolim_{Q_{\geq q}} \E)$ and $H_*(\hocolim_{f\inv(Q \setminus \{q\})} \D) \twoheadrightarrow H_*(\hocolim_{Q \setminus \{q\}} \E)$.  Since $\hocolim_P \D = \hocolim_{f\inv(Q_{\geq q})} \D \cup \hocolim_{f\inv(Q \setminus \{q\})} \D$ and $\hocolim_Q \E = \hocolim_{Q_{\geq q}} \E \cup \hocolim_{Q \setminus \{q\}} \E$, the result follows from Lemma \ref{GL}.
\end{proof}

\subsection{Topological Representations of Matroids}

In this section, we review Engstr\"{o}m's construction.  We begin with some more notation:

\begin{definition}
Given a CW complex $X$, an \emph{$X$-arrangement} consists of a CW complex $Y$ and a finite collection $\bA$ of subcomplexes of $Y$ such that: \begin{enumerate} \item The complex $Y$ is homotopy equivalent to $X^{*d}$ for some $d$, and $\dim(Y) = \dim(X^{*d})$. \item Each complex $A$ in $\bA$ is homotopy equivalent to $X^{*(d-1)}$ and $\dim(A) = \dim(X^{*(d-1)})$. \item Each intersection $B$ of complexes in $\bA$ is homotopy equivalent to some $X^{*e}$, and $\dim(B) = \dim(X^{*e})$. \item Every free action of a group $\Gamma$ on $X$ induces a free action of $\Gamma$ on $Y$ and every intersection of complexes in $\bA$. \item If $B \simeq X^{*e}$ is an intersection of complexes in $\bA$, the complex $A$ is in $\bA$, and $A \nsupseteq B$, then $A \cap B \simeq X^{*(e-1)}$. \end{enumerate} 
\end{definition}

Given an $X$-arrangement, one can always obtain the lattice of flats of a matroid in the following way.  A subset $\{A_1, A_2, ... , A_n\}$ of $\bA$ is a \emph{flat} in the arrangement if no subcomplex $B$ in $\bA \setminus \{A_1, A_2, ... , A_n\}$ contains $\displaystyle \bigcap_{i = 1}^{n} A_i.$

\begin{proposition}[\bf \cite{Engstrom10}]
If $(Y,\bA)$ is an $X$-arrangement, then the flats of $\bA$ are the flats of some matroid.
\end{proposition}

Engstr\"{o}m showed that the converse is also true.  His theorem consists of two parts:  (1) it states that any matroid can be represented as an $X$-arrangement for any ``nice'' CW complex $X$; and (2) it shows how to compute the homotopy type of the resulting space.  

\begin{theorem}[\bf Engstr\"om's Representation Theorem for Matroids, \cite{Engstrom10}]\label{TopRepMat}
Let $M$ be a rank $r$ matroid, and $l$ a rank- and order-reversing poset map from $\lat(M)$ to $B_r$, the boolean lattice on $[r]$.  Let $X$ be a locally finite, regular $CW$ complex and for each $\sigma \in B_r$ define $$D_X(\sigma) = *_{i=1}^r \begin{cases} X & \text{if $i \in \sigma$} \\ \emptyset & \text{if $i \notin \sigma$} \end{cases}$$ to get a $B_r$-diagram $\D_X$ with inclusion morphisms.  

Define the $\lat(M)$-diagram $\D_X(M,l) = \D_X \circ l$.  Then $$(Y,{\mathbf A}) = (\hocolim \D_X(M,l), \{\hocolim \D_X(M,l)_{\geq a} \ | \ a \text{ is an atom of } \lat(M)\})$$ is an $X$-arrangement of $\lat(M)$ and $$\Sph_X(M,l) := \bigcup_{A \in {\mathbf A}} A \simeq \bigvee_{p \in \lat(M) \setminus \hat{0}} \left( X^{*(r-\rk(p))} * \bigvee^{|\mu(\hat{0},p)|} S^{\rk(p) - 2} \right).$$  \end{theorem}

We call the space $\T_X(M,l)$ an \emph{Engstr\"om representation} of the pair $(M,l)$.  The regularity and finiteness conditions on $X$ are only needed for computing the homotopy type of $\T_X(M,l)$ with discrete Morse theory, \cite{Engstrom10}.  For the rest of the paper, we omit these conditions, but continue to assume $X$ is reasonably nice. 

\begin{remark}\label{SimpForm}
For each $p \in \lat(M) \setminus \hat{0}$, the space $X^{*(r-\rk(p))} * S^{\rk(p) - 2}$ depends only on the rank of $p$.  We rewrite the formula for the homotopy type of $\T_X(M,l)$ as follows: 
\begin{align*}
\T_X(M,l) &\simeq \bigvee_{p \in \lat(M) \setminus \hat{0}} \left( \bigvee^{|\mu(\hat{0},p)|} (X^{*(r-\rk(p))} * S^{\rk(p) - 2}) \right) \\ 
&\simeq \bigvee_{i = 1}^{r} \left(\bigvee^{w_i(M)} \left(X^{*(r-i)} * S^{i - 2}\right) \right) \\
&\simeq \bigvee_{i = 1}^{r} \left(\bigvee^{w_i(M)} \Susp^{i-1}\left(X^{*(r-i)}\right) \right)
\end{align*}
where $\Susp^k(Y)$ is the $k$-fold suspension of $Y$.
\end{remark}

\begin{example}\label{explicit}
Here we work out an explicit example for the construction in Theorem \ref{TopRepMat} on the matroid $M = ([5],\I)$ where $$\I = \left\{ \begin{matrix} \emptyset, \{1\}, \{2\}, \{3\}, \{4\}, \{5\}, \{1,3\}, \{1,4\}, \{1,5\}, \{2,3\}, \{2,4\}, \{2,5\} \\ \{3,4\}, \{3,5\}, \{4,5\}, \{1,3,5\}, \{1,4,5\}, \{2,3,5\}, \{2,4,5\}, \{3,4,5\} \end{matrix} \right\}.$$  The flats of $M$ are $$\F = \left\{ \begin{matrix} \emptyset, \{1,2\}, \{3\}, \{4\}, \{5\}, \{1,2,3,4\}, \{1,2,5\}, \{3,5\}, \{4,5\}, [5] \end{matrix} \right\}.$$  Since $M$ has rank three, we define a poset map $l : \lat(M) \to B_3$ by $l(\{1,2\}) = \{1,2\}$, $l(\{3\}) = l(\{4\}) = \{1,3\}$, and $l(\{5\}) = \{2,3\}$.  The Hasse diagram of $\lat(M)$ is drawn in Figure 1 decorated with triples of dots representing joins of spaces.  Shaded dots correspond to $X$ and unshaded dots correspond to $\emptyset$. 

\begin{figure}[h] \label{hasse_ex}
\begin{center} \includegraphics[width = 0.65\textwidth]{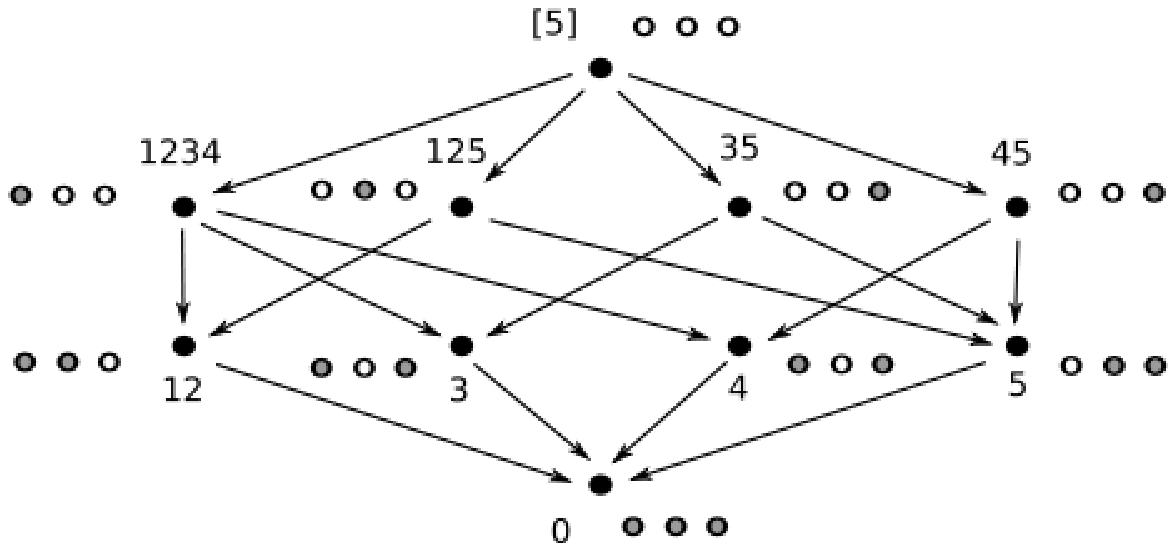} \end{center}
\caption{Decorated Hasse diagram of $\lat(M)$}
\end{figure}

The space $\T_X(M,l)$ with $X = S^0$ is drawn in Figure 2.
 
\begin{figure}[!h]
\begin{center} \includegraphics[width = 0.45\textwidth]{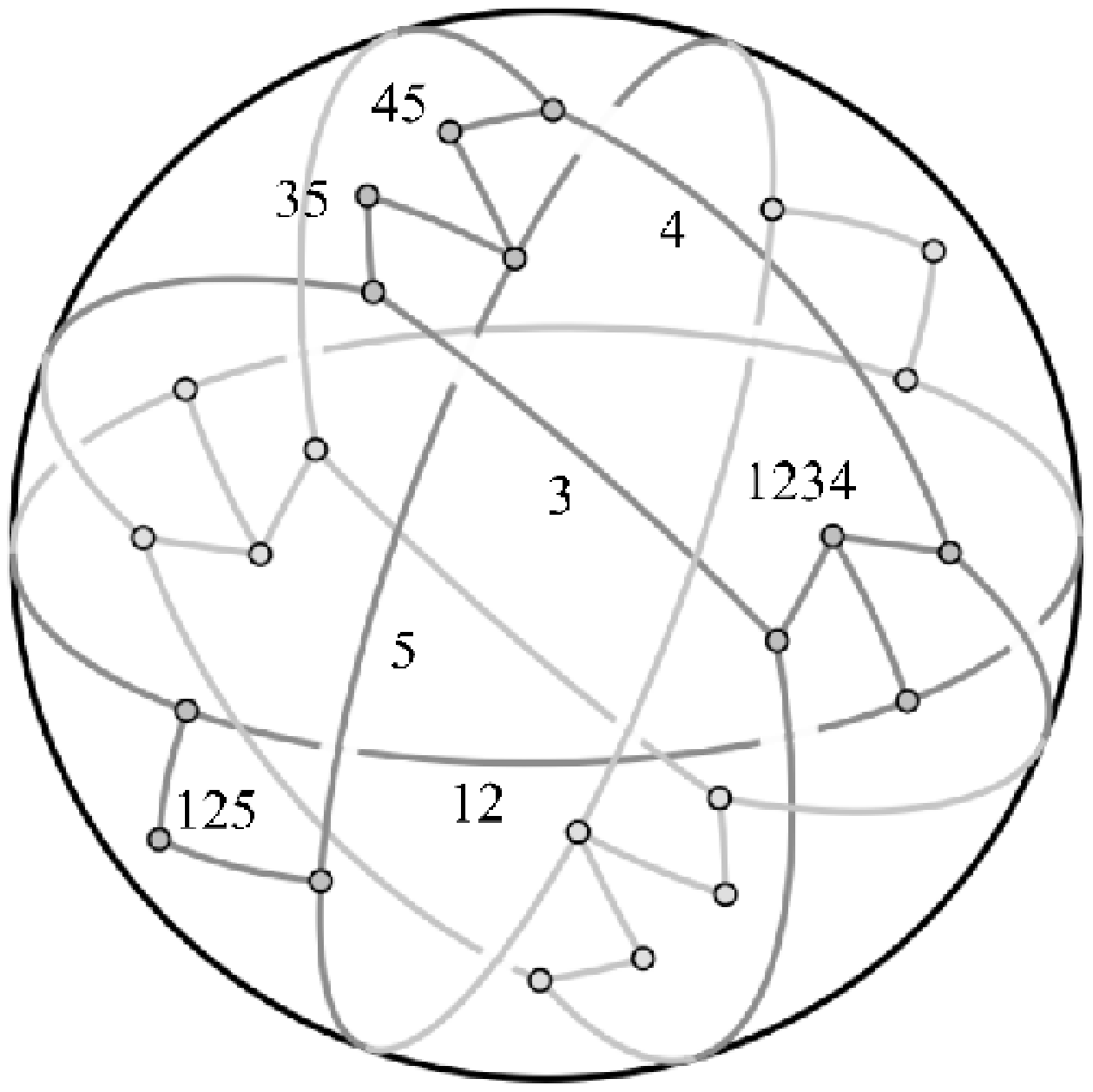} \hspace{0.05\textwidth} \includegraphics[width = 0.45\textwidth]{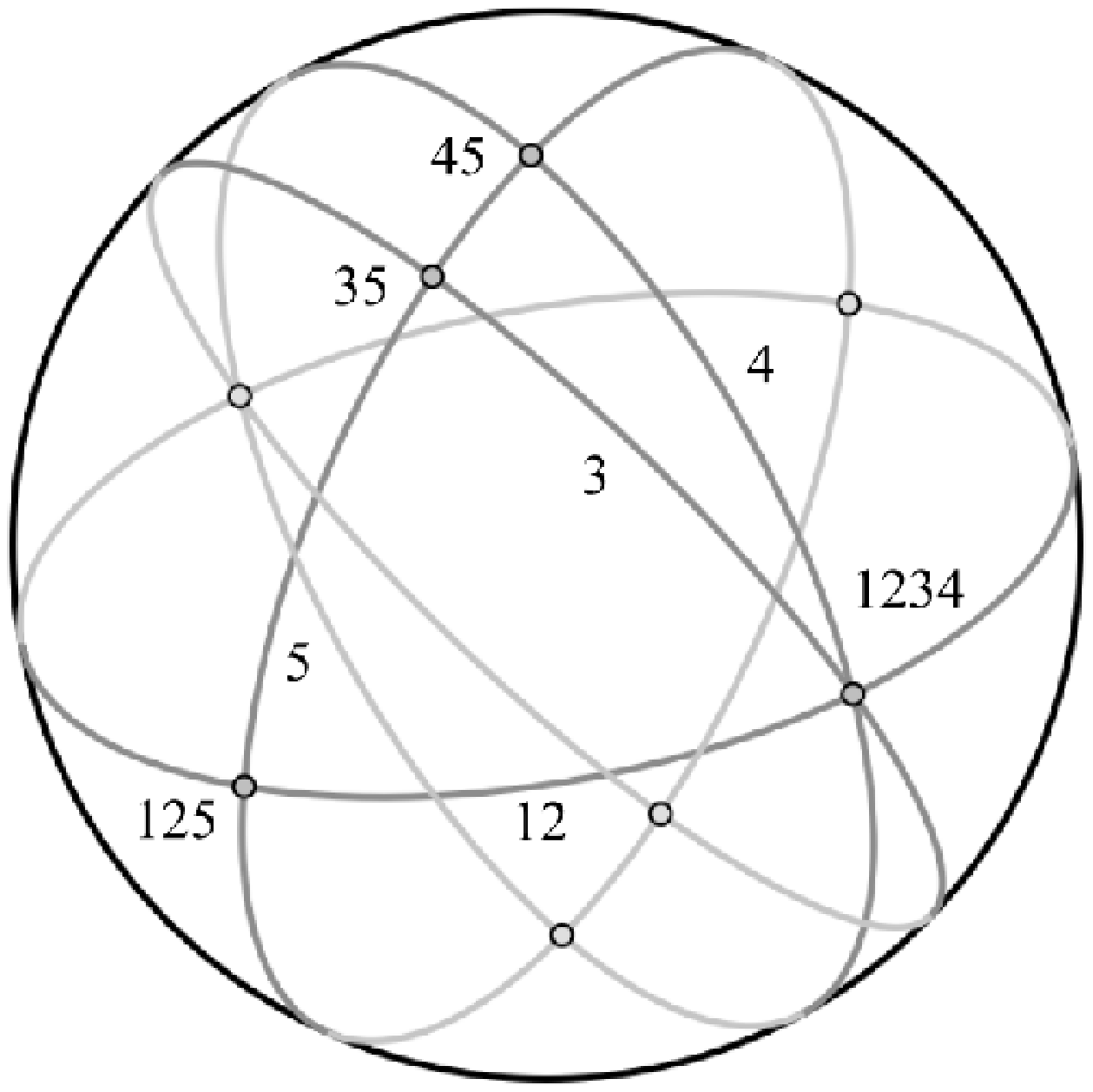} \end{center}
\caption{The space $\T_{S^0}(M,l)$ from Example \ref{explicit} (left) compared with a spherical arrangement whose lattice of flats matches that of $M$ (right)}
\end{figure} 
\end{example}

\begin{example}\label{real/complex}
Here we compute the homotopy types of $\T_{S^0}(U_{r,n})$ and $\T_{S^1}(U_{r,n})$, which are codimension one and codimension two homotopy sphere arrangements representing $U_{r,n}$, the uniform rank $r$ matroid on the set $[n]$.  By Theorem \ref{TopRepMat} and Remark \ref{SimpForm}, we have $$\T_{S^0}(U_{r,n}) \simeq \bigvee_{i = 1}^{r} \left(\bigvee^{w_i(U_{r,n})} \Susp^{i-1}\left((S^0)^{*(r-i)}\right) \right) \simeq \bigvee_{i=1}^r \left( \bigvee^{w_i(U_{r,n})} S^{r - 2} \right)$$ and  
$$\T_{S^1}(U_{r,n}) \simeq \bigvee_{i = 1}^{r} \left(\bigvee^{w_i(U_{r,n})} \Susp^{i-1}\left((S^1)^{*(r-i)}\right) \right) \simeq \bigvee_{i=1}^r \left( \bigvee^{w_i(U_{r,n})} S^{2r - i - 2} \right).$$  Therefore, the representations are homotopy equivalent to wedges of spheres.  To count the number of spheres of each dimension, we compute $w_i(U_{r,n})$ for each $i \in [r]$.  Observe that the subposet of $\lat(U_{r,n})$ consisting of elements with rank at most $r-1$ is isomorphic to that of $B_n$.  So, for every $p \in \lat(U_{r,n})$ with $\rk(p) \leq r-1$, $\mu(\hat{0},p) = (-1)^{\rk(p)}$ and hence, $w_i(U_{r,n}) = \binom{n}{i}$ for each $i \in [r-1]$.  The only remaining element of $\lat(U_{r,n})$ is the single rank $r$ flat $[n]$.  To compute $w_r(U_{r,n}) = |\mu(\hat{0},[n])|$, we apply Proposition \ref{enum_lem}.  Let $a \in [n]$.  The coatoms of $\lat(U_{r,n})$ which do not cover $a$ are the $r-1$ element subsets of $[n]$ which do not contain $a$.  There are $\binom{n-1}{r-1}$ of them.  Thus, $$\T_{S^0}(U_{r,n}) \simeq \bigvee^{m} S^{r - 2} \ \ \text{ and } \ \ \T_{S^1}(U_{r,n}) \simeq \bigvee_{i = 1}^r \left( \bigvee^{m(i)} S^{2r - i - 2} \right)$$ where $m = \binom{n-1}{r-1} + \sum\limits_{k = 1}^{r-1} \binom{n}{k}$ and $m(i)$ equals $\binom{n}{i}$ if $i \in [r-1]$ and $\binom{n-1}{i-1}$ if $i = r$.  
\end{example}

The space $\T_{S^0}(M,l)$ is the homotopy sphere arrangement analog of a real hyperplane arrangement and $\T_{S^1}(M,l)$ is the homotopy sphere arrangement analog of a complex hyperplane arrangement.  Example \ref{real/complex} illustrates that this analogy is drawn not only from the codimension of the arrangements, but from their overall behavior as well.  The codimension two picture reveals more information about the underlying matroid than the codimension one case, since each nonzero Betti number of $\T_{S^1}(M,l)$ is determined by the elements of $\lat(M)$ with a specific rank whereas $\beta_{\rk(M)-2}$ is the only nontrivial Betti number of $\T_{S^0}(M,l)$.

\section{Topological Representations of Matroid Maps}

This is the main part of the paper, in which we prove that weak maps between matroids induce continuous maps between their topological representations.  We also observe, from the theory of homotopy colimits, that the induced maps are $\Gamma$-equivariant whenever the CW complex $X,$ in the $X$-arrangements, supports a free $\Gamma$-action.  We finish the section with a new and conceptual proof of Proposition \ref{whit_dec} and show that isomorphisms are the only strong maps which induce homotopy equivalences.  

An enticing feature of Theorem \ref{TopRepMat} is the specification of a rank- and order-reversing poset map $l : \lat(M) \to B_{r}$.  This allows us to obtain different geometric realizations of the same representation in a combinatorially convenient way, adding versatility to the construction in a number of applications.  For instances in which the specific geometric realization on $\T_X(M,l)$ is less vital, we observe that the homotopy type of $\T_X(M,l)$ is independent of the choice of $l$.  In such cases, we often specify a choice for each map.  Either way, maintaining this versatility with respect to matroid maps requires a bit of bookkeeping.    We proceed with some new definitions:

\begin{definition}
Let $M$ be a matroid and $\rho \in \N$ such that $\rho \geq \rk(M)$.  A \emph{$\rho$-immersion} of $M$ is a rank- and order-reversing poset map $l : \lat(M) \to B_{\rho}$.
\end{definition}

For convenience, we refer to the pair $(M,l)$ where $M$ is a matroid and $l$ is a $\rho$-immersion of $M$ as a \emph{$\rho$-immersed matroid}.  In the case where $\rho = \rk(M)$, we abbreviate $\rho$-immersed matroid to \emph{immersed matroid}. 

\begin{definition}  Given $\rho$-immersions $l:M \to B_{\rho}$ and $l':N \to B_{\rho}$ of matroids $M$ and $N$, a weak map $\tau : M \to N$ is called \emph{$\rho$-admissible} if for every $p \in \lat(M)$, $l(p) \subseteq l'(\tau^{\#}(p))$. 
\end{definition}

Before we address the main theorems, we make two remarks:  the first asserts that results about $\rho$-immersed matroids with $\rho$-admissible maps apply to matroids and weak maps; the second describes how the topological picture changes when we consider the topological representation of a $\rho$-immersion of a matroid $M$ where $\rho > \rk(M)$. 

\begin{remark}
For any matroid $M$ there are, in general, many choices of $l$ which yield an immersed matroid, see \cite{Engstrom10}.  Moreover, for any weak map $\tau : M \to N$ and $\rho \geq \max\{\rk(M),\rk(N)\}$, one can always find a pair of maps $l$ and $l'$ so that $\tau : (M,l) \to (N,l')$ is $\rho$-admissible.  
\end{remark}

Indeed, one can fix a canonical map $\hat{l} : \lat(M) \to B_{\rho}$ for all matroids $M$ with $\rho \geq \rk(M)$. For the rest of the paper, we set $\hat{l} : \lat(M) \to B_{\rho}$ to be the map defined by $\hat{l}(p) = [\rho - \rk(p)]$ for each matroid $M$.  When using $\hat{l}$, we simplify our notation to $\D_X(M) := \D_X(M,\hat{l})$ and $\T_X(M) := \T_X(N,\hat{l})$.

\begin{lemma}\label{EZ}
Every weak map $\tau : M \to N$ between matroids $M$ and $N$ is a $\rho$-admissible weak map between the $\rho$-immersed matroids $(M,\hat{l})$ and $(N,\hat{l})$ for each $\rho \geq \max \{\rk(M),\rk(N)\}$. 
\end{lemma}

\begin{proof}  Let $p \in \lat(M)$.  Since $\tau$ is a weak map, $\rk_N(\tau^{\#}(p)) \leq \rk_M(p)$ and hence, $$\hat{l}(p) = [\rho-\rk_M(p)] \subseteq [\rho-\rk_N(\tau^{\#}(p))] = \hat{l}(\tau^{\#}(p)).$$  Thus, $\tau$ is $\rho$-admissible. \end{proof}

\begin{remark}
For a $\rho$-immersed matroid $(M,l)$ with $\rho > \rk(M)$, the space $\T_X(M,l)$ is not an $X$-arrangement. 
\end{remark}

We leave it to the reader to check that the fifth condition in the definition of an $X$-arrangement is not satisfied.  It turns out, however, that these spaces are merely $X$-suspensions of $X$-arrangements.

\begin{proposition}\label{stable_prop}
For any $\rho$-immersion $l$ of a matroid $M$, $\T_X(M,l) \simeq X^{\rho - \rk(M)} * \T_X(M)$.
\end{proposition}

\begin{proof}
It is a straightforward fact about homotopy colimits that for any $P$-diagram $\D$ and finite CW complex $Y$, $Y * \hocolim_P \D = \hocolim_P \E$ where $\E = Y * \D$ is the diagram defined by associating $Y * D_p$ to each $p \in P$ and letting $id * d_{pq} : Y*D_p \to Y*D_q$ be the map induced by taking the identity on $Y$, $d_{pq} : D_p \to D_q$, and extending linearly to each element in the join $Y * D_p$ for each $p \geq q$ in $P$.  Thus, $X^{\rho - \rk(M)} * \T_X(M) = \hocolim_{\lat(M) \setminus \hat{0}} X^{\rho-\rk(M)} * \D_X(M)$.  Since $D_X(p) \simeq X^{*(\rk(M) - \rk(p))}$ for each $p \in \lat(M)$, we know that $X^{\rho-\rk(M)} * D_X(p) \simeq X^{\rho - \rk(p)} \simeq D_X(l(p))$.  The result follows from Lemma \ref{HL}.   
\end{proof}

Since the dimension of $\T_X(M)$ is a function of $\rk(M)$, it is not always possible to find nice maps between $\T_X(M)$ and $\T_X(N)$ when $\rk(M) \neq \rk(N)$.  In such cases, we can still make \emph{stable} statements, i.e. $\T_X(M)$ and $\T_X(N)$ are compared after a sufficient number of $X$-suspensions.  In the language of Proposition \ref{stable_prop}, this corresponds to taking $\rho$-immersions $l$ and $l'$ of $M$ and $N$, respectively, for some common $\rho \in \N$ and comparing $\T_X(M,l)$ to $\T_X(N,l')$.  We are now ready to present the main theorems of the paper.

\subsection{Continuous and Equivariant Representations}    

We begin by showing that $\rho$-admissible weak maps induce morphisms of diagrams.

\begin{theorem}\label{weak_map_thm}
Let $X$ and $Y$ be finite CW complexes with a continuous map $f : X \to Y$ and let $(M,l)$ and $(N,l')$ be $\rho$-immersed matroids where $\rho \geq \max \{\rk(M),\rk(N)\}$.  If $\tau : (M,l) \to (N,l')$ is a $\rho$-admissible weak map, then $\tau^{\#}$ induces a morphism of diagrams $$(\tau^{\#},\alpha) : \D_X(M,l) \to \D_Y(N,l').$$
\end{theorem}

\begin{proof}
The map $\tau^{\#}$ is clearly a functor from $\lat(M)$ to $\lat(N)$.  We proceed by constructing a natural transformation $\alpha$ from $\D_X(M,l)$ to $\D_Y(N,l') \circ \tau^{\#}$. Since $\tau$ is $\rho$-admissible, we get a natural transformation $\eta : l \to l' \circ \tau^{\#}$ where $\eta_p$ is the ``$\subseteq$'' map in $B_{\rho}$.  Next, let $f_{s} : D_X(s) \to D_Y(s)$ be the diagonal map induced by $f$ for each $s \in B_{\rho}$.  Then $f$ yields a natural transformation $\bar{f}$ from $\D_X$ to $\D_Y$. This can be seen diagrammatically since for each $p \geq q$ in $\lat(M)$ and $s \subseteq t$ in $B_{\rho}$, $$\begin{matrix} & & \subseteq & & \\ & l(p) & \longrightarrow & l(q) & \\ \eta_p & \downarrow & & \downarrow & \eta_q \\ & l'(\tau^{\#}(p)) & \longrightarrow & l'(\tau^{\#}(q)) & \\ & & \subseteq & & \end{matrix} \hspace{0.5in} \begin{matrix} & & \iota & & \\ & D_X(s) & \longrightarrow & D_X(t) & \\ f_{s} & \downarrow & & \downarrow & f_{t} \\ & D_Y(s) & \longrightarrow & D_Y(t) & \\ & & \iota & & \end{matrix}$$ clearly commute.  Since $\D_X(M,l) = \D_X \circ l$ and $\D_Y(N,l') \circ \tau^{\#} = \D_Y \circ l' \circ \tau^{\#}$, define the natural transformation $\alpha$ to be the ``horizontal'' composition of $\eta$ and $\bar{f}$.  
\end{proof}

Since morphisms of diagrams induce continuous maps between their homotopy colimits, Theorem \ref{weak_map_thm} immediately implies that non-annihilating, admissible weak maps induce continuous maps between the topological representations of immersed matroids.  We can prove a much stronger corollary using the theory of homotopy colimits.

\begin{corollary}\label{cont_map_thm}
For any finite CW complexes, $X$ and $Y$, and continuous map ${f : X \to Y}$, a $\rho$-admissible weak map $\tau : (M,l) \to (N,l')$ between $\rho$-immersed matroids induces a continuous map $$\tau^* : \hocolim_{\lat(M)} \D_X(M,l) \to \hocolim_{\lat(N)} \D_Y(N,l')$$ such that (1) $\tau^*|_{\T_X(M,l)}$ is homotopic to a map into $\T_Y(N,l')$ and (2) if a group $\Gamma$ acts freely on $X$ and $Y$ such that $f$ is $\Gamma$-equivariant, then $\tau^*$ is $\Gamma$-equivariant as well. 
\end{corollary}

\begin{proof}
The existence of $\tau^*$ follows directly from Theorem \ref{weak_map_thm} and Proposition \ref{morphdiag}.  If $\tau$ is non-annihilating, then $\tau^*|_{\T_X(M,l)}$ maps into $\T_Y(N,l')$ and (1) is trivial.  Otherwise, let $a$ be any atom in $\tau^{\#}(\lat(M) \setminus \hat{0})$ and let $\tau_a : \lat(M) \to \lat(N)$ be the map induced by $$\tau_a(p) := \begin{cases} \tau^{\#}(p) & \text{ if $\tau^{\#}(p) \neq \hat{0} \in \lat(N)$ } \\ a & \text{ otherwise }  \end{cases}$$ for every atom $p \in \lat(M)$.  By Proposition \ref{morph_hom}, $\tau^* \simeq \tau_a^*$ and since $\tau_a$ is non-annihilating, $\tau_a^*$ maps $\T_X(M,l)$ to $\T_Y(N,l')$. Condition (2) falls out naturally from \cite{DH}. 
\end{proof}

While Theorem \ref{weak_map_thm} and Corollary \ref{cont_map_thm} are rather general, they restrict to special cases which are more practical.  For instance, if we set $X = Y$, choose a canonical family of immersions, such as $\hat{l}$, and restrict our attention to matroids of a fixed rank, then we get the following simpler statements.

\begin{corollary}
Let $X$ be a finite CW complex and let $M$ and $N$ be rank $r$ matroids.  If $\tau : M \to N$ is a weak map, then $\tau^{\#}$ induces a morphism of diagrams $(\tau^{\#},\iota) : \D_X(M) \to \D_X(N)$ where $\iota$ is the natural transformation given by inclusion maps.
\end{corollary}

\begin{corollary} 
For any finite CW complex, $X$, a non-annihilating weak map $\tau : M \to N$ between rank $r$ matroids induces a continuous map $\tau^* : \T_X(M) \to \T_X(N).$  Moreover, if there is a free action of a group $\Gamma$ on $X$, then $\tau^*$ is $\Gamma$-equivariant. 
\end{corollary}

\subsection{Properties of Continuous Representations}\label{Apps}

In this section, we explore several combinatorial aspects of weak maps which have nice topological interpretations.  We begin with a new and conceptual proof of Proposition \ref{whit_dec}.  For a fixed CW complex $X$ and $\rho \in \N$, define $Y_i = \Susp^{i-1}(X^{*(\rho-i)})$ for each $1 \leq i \leq \rho$.  By Remark \ref{SimpForm}, for any $\rho$-immersed matroid $(M,l)$, \begin{equation}\tag{\textasteriskcentered} \beta_k(\T_X(M,l)) = \sum\limits_{i=1}^{\rk(M)} w_i(M) \beta_k(Y_i)  \end{equation} for each $k \in \N$.  So, there is a natural correspondence between the Whitney numbers of the first kind of a matroid and the Betti numbers of its Engstr\"om representations.

\begin{theorem}\label{betti_dec}
Let $X$ be a finite CW complex and $\tau$ be a surjective $\rho$-admissable weak map between $\rho$-immersed matroids $(M,l)$ and $(N,l')$. Then $$\beta_k(\T_X(M,l)) \geq \beta_k(\T_X(N,l'))$$ for each $k \in \N$.
\end{theorem}

\begin{proof}
We prove that $\tau^*$ induces a surjection in homology.  The key observation is that the space $\hocolim_{\lat(N)_{\geq q}} \D_X$ contracts onto $ D_X(q)$ for each $q \in \lat(N)$.  By Lemma \ref{surj_poset}, there exists a flat $p' \in {(\tau^{\#})}\inv(q)$ with $\rk_M(p') = \rk_N(q)$.  Thus, $\hocolim_{\{p \in P : \tau^\#(p) \geq q\}} \D$ contains $D_X(p')$, which is included into $D_X(q)$ by $\tau^*$, and hence, $\tau^*$ induces a surjection $$H_*(\hocolim_{\{p \in \lat(M) : \tau^\#(p) \geq q\}} \D_X) \twoheadrightarrow H_*(\hocolim_{\lat(N)_{\geq q}} \D_X).$$  By Lemma \ref{VUFL}, $\tau^*$ induces a surjection $H_*(\T_X(M,l)) \twoheadrightarrow H_*(\T_X(N,l'))$ and the result follows immediately.
\end{proof}

When $X \simeq S^1$, $Y_i \simeq S^{2r - i + 2}$ and thus, $\beta_{2r-i+2}(\T_{S^1}(M)) = w_i(M)$ for each $i \in \{0,1,...,r\}$.  Proposition \ref{whit_dec} is now an easy corollary of Theorem \ref{betti_dec}:

\begin{corollary}[\bf Proposition \ref{whit_dec}]
If $\tau : M \to N$ is a surjective weak map and $\rk(M) = \rk(N) = r$, then $w_k(M) \geq w_k(N)$ for each $k \in \{0,...,r\}$.
\end{corollary}

\begin{proof}
For $k \in \{0,...,r\}$, $w_k(M) = \beta_{2r - k + 2}(\T_{S^1}(M)) \geq \beta_{2r - k + 2}(\T_{S^1}(N)) = w_k(N).$
\end{proof}

Next, we take a close look at truncations.  To begin, observe that for any rank $r$ matroid $M$ and $n \leq r$, the lattices $\lat(M)$ and $\lat(T^{r-n}(M))$ are isomorphic up to rank $n-1$.  Thus, $w_k(M) = w_k(T^{r-n}(M))$ for each $k \leq n-1$.  For $k = n$, we have the following lemma which we suspect is already known. Since we could not find a reference for it, we include a proof here.

\begin{lemma} \label{trunc_lem}
For any rank $r$ matroid $M$ and $n < r$, $w_n(M) \geq w_n(T^{r-n}(M)).$
\end{lemma}

\begin{proof}
Let $a$ be a rank one flat of  $T^{r-n}(M)$, let $\F_a^{n-1}$ be the set of rank $n-1$ flats of $T^{r-n}(M)$ which do not cover $a$, and let $\F^n$ be the set of rank $n$ flats of $M$.  By Lemma \ref{enum_lem}, $$w_n(T^{r-n}(M)) = \mu_{\lat(T^{r-n}(M))}(\hat{0},\hat{1}) = \sum\limits_{q \in \F_a^{n-1}} |\mu_{\lat(T^{r-n}(M))}(\hat{0},q)|.$$  Since every interval in a geometric lattice is geometric, we apply Lemma \ref{enum_lem} again to get $\mu_{\lat(M)}(\hat{0},p) = - \sum\limits_{\{q \in F_a^{n-1} \, | \, q \leq p\}} \mu_{\lat(M)}(\hat{0},q)$ for each $p \in \F^n$.  Observe that each $q \in \F_a^{n-1}$ is covered by at least one element of $\F^n$, so $$w_n(M) = \sum\limits_{p \in \F^n} | \mu_{\lat(M)}(\hat{0},p) | \geq  \sum\limits_{q \in \F_a^{n-1}} |\mu_{\lat(M)}(\hat{0},q)| =  w_n(\lat(T^{r-n}(M))),$$ because $\mu_{\lat(M)}(\hat{0},q) = \mu_{\lat(T^{r-n}(M))}(\hat{0},q)$ for all $p$ with $\rk(q) < n$.  \end{proof}

We can now give some sufficient conditions for when the continuous representation of a surjective weak map strictly decreases Betti numbers.

\begin{theorem} \label{betti_weak}
Let $X$ be a finite CW complex and $\tau$ be a surjective $\rho$-admissable weak map between $\rho$-immersed matroids $(M,l)$ and $(N,l')$.  If $\rk(M) > \rk(N)$ and $\beta_k(Y_i) \neq 0$ for some $i \in \{\rk(N)+1, ... , \rk(M)\}$, then $\beta_k(\T_X(M,l)) > \beta_k(\T_X(N,l'))$.
\end{theorem} 

\begin{proof}
By Lemma \ref{weak_fact}, $\tau$ factors uniquely through a truncation.  Thus, Proposition \ref{whit_dec} and Lemma \ref{trunc_lem} combine to give $w_i(M) \geq w_i(N)$ for all $i \in [\rk(N)]$.  By Equation (\textasteriskcentered), \begin{align*}\beta_k(\T_X(M,l)) &= \sum\limits_{i=1}^{\rk(N)} w_i(M) \beta_k(Y_i) + \sum\limits_{i=\rk(N)+1}^{\rk(M)} w_i(M) \beta_k(Y_i) \\ &\geq \sum\limits_{i=1}^{\rk(N)} w_i(N) \beta_k(Y_i) + \sum\limits_{i=\rk(N)+1}^{\rk(M)} w_i(M) \beta_k(Y_i) \\ &= \beta_k(\T_X(N,l')) \end{align*}  Since each $w_i$ is positive (Lemma \ref{pos_lem}), $w_i(M) \beta_k(Y_i)$ is positive whenever $\beta_k(Y_i) \neq 0$. 
\end{proof}

When $X = S^0$, $Y_i \simeq S^{\rho - 2}$ for each $1 \leq i \leq \rho$.  Since $\T_{S^0}(M,l)$ is a wedge of $(\rho - 2)$-spheres, $\beta_{\rho-2}$ is the only nonzero Betti number.  

\begin{corollary}
If $\tau$ is a surjective $\rho$-admissable weak map between $\rho$-immersed matroids $(M,l)$ and $(N,l')$ and $\rk(M) > \rk(N)$, then $\beta_{\rho - 2}(\T_{S^0}(M,l)) > \beta_{\rho - 2}(\T_{S^0}(N,l')).$ 
\end{corollary}

We conclude this section by considering when the Engstr\"om representations of two matroids on the same ground set, which are comparable by weak order, can be homotopy equivalent.  By experimentation, this only seems to happen when the matroids have isomorphic lattices of flats or equivalently, have the same simplification.

\begin{conjecture}
Given a finite CW complex $X$ which is not contractible and two $\rho$-immersed matroids $(M,l)$ and $(N,l')$ such that $\T_X(M,l) \simeq \T_X(N,l')$, if there exists a surjective weak map $\tau : M \to N$, then $\tau^{\#}$ is an isomorphism.
\end{conjecture}

We are able to give a partial result in the case of strong maps.

\begin{proposition}\label{strong_prop}
Given a finite CW complex $X$ which is not contractible and two $\rho$-immersed matroids $(M,l)$ and $(N,l')$ such that $\T_X(M,l) \simeq \T_X(N,l')$, if there exists a surjective strong map $\sigma : M \to N$, then $\sigma$ is an isomorphism.
\end{proposition}

\begin{proof}
The main observation is that $M$ and $N$ must have the same rank.  If $\rk(M) > \rk(N)$, then by Theorem \ref{betti_weak}, there is a $k \in \N$ such that $\beta_k(\T_X(M,l)) > \beta_k(\T_X(N,l'))$ which is not possible since $\T_X(M,l) \simeq \T_X(N,l')$.  Therefore, $\rk(M) = \rk(N)$ and, by Proposition \ref{strong_cor}, $M \cong N$.
\end{proof}

Whenever $X$ is a homotopy sphere, the conditions of Proposition \ref{strong_prop} are satisfied.  Therefore, any mapping between homotopy sphere arrangements which arises from a strong map, but is not an isomorphism of the underlying matroids, strictly decreases Betti numbers.

\section{Functoriality}

Let $\M(r,n)$ be the category of rank $r$ matroids on $n$ elements with weak maps and define $\OA(r,n)$ to be the category of immersed rank $r$ matroids on $n$ elements with admissible weak maps.  In this section, we show that $\T_X : \OA(r,n) \to \hTop$ is a functor for each CW complex $X$, and that it extends naturally to a functor $\M(r,n) \to \hTop$.  Here, $\hTop$ denotes the homotopy category of $\Top$, i.e. the category with topological spaces as objects and homotopy classes of continuous maps as morphisms. 

By Theorem \ref{TopRepMat}, we know $\T_X$ maps the objects of $\OA(r,n)$ to objects in $\Top$ and by Corollary \ref{cont_map_thm}, $\T_X$ maps morphisms of $\OA(r,n)$ to homotopy classes of morphisms in $\Top$.  It is not immediately clear, however, that this correspondence is functorial since composition is not always preserved by $\#$.  

\begin{example} \label{func_ex}Let $M,N,L \in \M(3,4)$ be defined by their flats $$\begin{matrix} \F_M = \left\{ \begin{matrix} \emptyset, \{1\}, \{2\}, \{3\}, \{4\}, \{1,2\}, \{1,3\}, \{1,4\}, \{2,3\}, \{2,4\}, \{3,4\}, [4] \end{matrix} \right\} \\ \F_N = \left\{ \begin{matrix} \emptyset, \{1\}, \{2\}, \{3\}, \{4\}, \{1,2\}, \{1,3\}, \{1,4\}, \{2,3,4\}, [4] \end{matrix} \right\} \\ \F_L = \left\{ \begin{matrix} \emptyset, \{1\}, \{2\}, \{3,4\}, \{1,2\}, \{1,3,4\}, \{2,3,4\}, [4] \end{matrix} \right\} \end{matrix}$$ and observe that the identity map on $[4]$ gives weak maps $id_{MN} : M \to N$, $id_{NL} : N \to L$, and $id_{ML}: M \to L$.  Clearly, $id_{ML} = id_{NL} \circ id_{MN}$, but $$id_{ML}^{\#}(\{3,4\}) = \{3,4\} \neq \{2,3,4\} = id_{NL}^{\#} \circ id_{MN}^{\#}(\{3,4\}).$$

\begin{figure}[!h]
\begin{center} 
\includegraphics[width = 0.9\textwidth]{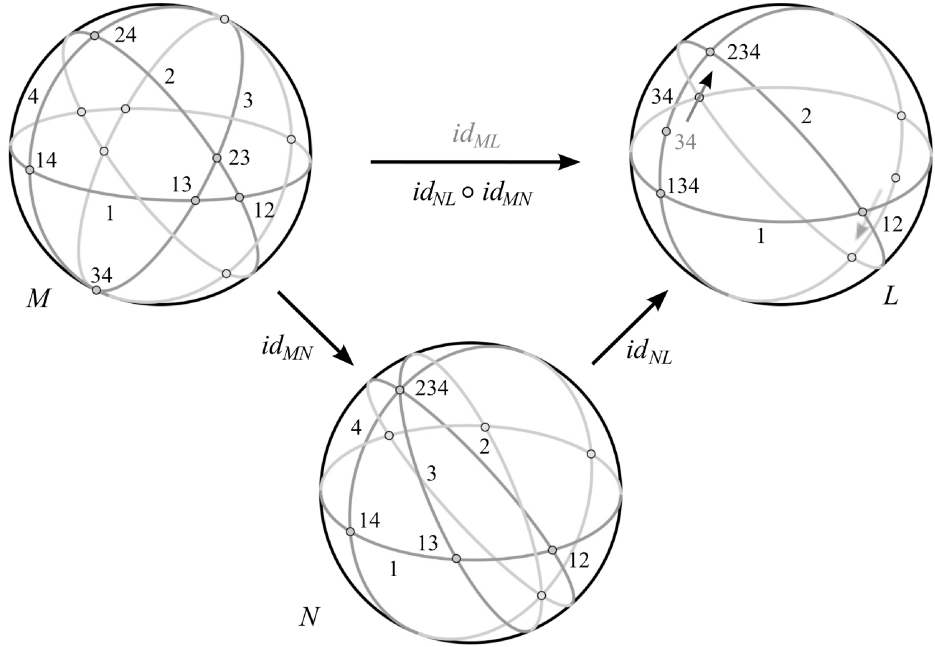}
\end{center}
\caption{Sphere arrangements corresponding to matroids $M$, $N$, and $L$ in Example \ref{func_ex}}
\end{figure}

\end{example}

Fortunately, we can show that $\T_X$ preserves composition up to homotopy, i.e. $(\sigma \circ \tau)^* \simeq \sigma^* \circ \tau^*$ for any pair of weak maps $\tau : M \to N$ and $\sigma : N \to L$.  This is illustrated in Figure 3 by the arrows mapping $id_{ML}^{\#}(\{3,4\}) = \{3,4\}$ to $id_{NL}^{\#} \circ id_{MN}^{\#}(\{3,4\}) = \{2,3,4\}$.   

\begin{theorem}\label{functor} 
For any fixed $r,n \in \N$ with $r \leq n$ and $CW$ complex $X$, $\T_X$ is a functor from \emph{$\OA(r,n)$} to $\hTop$.  Furthermore, if $X$ is equipped with a free $\Gamma$-action for some group $\Gamma$, then $\T_X^{\Gamma}$ is a functor from \emph{$\OA(r,n)$} to $\Gamma$-$\hTop$. 
\end{theorem}  

\begin{proof}  From the preceding discussion, we know that $\T_X$ maps objects to objects and morphisms to morphisms.  All that remains to check is that $\T_X$ preserves identity maps and compositions up to homotopy.  To see that the identity map $id$ on an immersed matroid $(M,l)$ induces the identity map on $\T_X(M,l)$, observe that $id^{\#}$ is the identity on $\lat(M)$ and that each component of $\alpha$ in the morphism $(id,\alpha) : \D_X(M,l) \to \D_X(M,l)$ is the identity map.  Therefore, $(id,\alpha)$ is the identity morphism on $\D_X(M,l)$ which induces the identity map on its homotopy colimit, and hence on $\T_X(M,l)$.

Next, let $\tau : (M,l_M) \to (N,l_N)$ and $\sigma : (N,l_N) \to (L,l_L)$ be admissible weak maps between rank $r$ immersed matroids.  We claim that $(\sigma \circ \tau)^{\#}(p) \subseteq \sigma^{\#} \circ \tau^{\#}(p)$ for all $p \in \lat(M)$.  Clearly, $\tau(p) \subseteq cl_N(\tau(p))$, so $\sigma(\tau(p)) \subseteq \sigma(cl_N(\tau(p)))$.  This implies that $$(\sigma \circ \tau)^{\#}(p) = cl_L(\sigma(\tau(p))) \subseteq cl_L(\sigma(cl_N(\tau(p)))) = \sigma^{\#} \circ \tau^{\#}(p)$$ for all $p \in \lat(M)$.  By Remark \ref{morph_hom_rem}, there exists a natural transformation $\gamma : \sigma^{\#} \circ \tau^{\#} \to (\sigma \circ \tau)^{\#}$.  Consider the morphisms of diagrams $(\sigma^{\#} \circ \tau^{\#},\alpha)$ and $((\sigma \circ \tau)^{\#},\beta)$ obtained in Theorem \ref{weak_map_thm}, i.e. $\alpha$ and $\beta$ are natural transformations whose components consist of inclusion maps.  By Proposition \ref{morph_hom}, we know that $(\sigma \circ \tau)^* \simeq \sigma^* \circ \tau^*$ and therefore, $\T_X$ is a functor from $\OA(r,n)$ to $\hTop$.  The $\Gamma$-equivariant version of the theorem falls out naturally from the homotopy colimit toolbox in \cite{DH}.
\end{proof}

By Lemma \ref{EZ}, we can immediately restate Theorem \ref{functor} in terms of matroids.

\begin{corollary}
The map $\mu : \M(r,n) \to \emph{$\OA(r,n)$}$ defined by $M \mapsto (M,\hat{l})$ extends each $\T_X$ to a functor from $\M(r,n)$ to $\hTop$.
\end{corollary}

We conclude this section with the following remark:

\begin{remark} At first glance, one might hope for a strengthening of Corollary \ref{cont_map_thm} that yields a functorial relationship from $\OA(r,n)$ to $\Top$ rather than $\hTop$.  Unfortunately, even if such an improvement is possible, the fact that $\#$ does not preserve composition precludes the existence of such a functor.  Thus, Theorem \ref{functor} is the best we can possibly hope for in our context. \end{remark}

\section{Future Directions}

This paper establishes a new framework for studying the structure theory of matroids regarding weak and strong maps.  We conclude with some unresolved questions and a discussion potential areas where our work may be beneficial.

\subsection{The Choice of Immersion of a Matroid}

In Theorem \ref{TopRepMat}, we get that the homotopy type of $\T_X(M,l)$ is independent of the choice of $l$.  It would be nice to have a more direct proof of this fact.

\begin{question}
Given two immersions $(M,l)$ and $(M,l')$ of the same matroid $M$, can we find an \emph{explicit} map between $\T_X(M,l)$ and $\T_X(M,l')$ which is a homotopy equivalence?
\end{question}

\subsection{The Unimodality of the Whitney Numbers}

With the new topological interpretation of Whitney numbers presented in this paper, one may hope to solve more elusive problems such as the famous unimodality conjectures of Rota \cite{Oxley,Rota71}.

\begin{conjecture}[Rota, 1971] \label{whit_conj}  For every rank $r$ geometric lattice, $L$, the Whitney numbers of the first kind are unimodal, i.e. for some $0 \leq k \leq r$, $$w_0(L) \leq \dots \leq w_{k-1}(L) \leq w_k(L) \geq w_{k+1}(L) \geq \dots \geq w_{r}(L).$$ \end{conjecture}

\begin{conjecture}[Rota, 1971]  For every rank $r$ geometric lattice, $L$, the Whitney numbers of the second kind are unimodal, i.e. for some $0 \leq k \leq r$, $$W_0(L) \leq \dots \leq W_{k-1}(L) \leq W_k(L) \geq W_{k+1}(L) \geq \dots \geq W_{r}(L).$$ \end{conjecture}

Huh and Katz recently proved Conjecture \ref{whit_conj} for representable matroids by studying the intersection theory of certain toric varieties \cite{HK11}.  Since homotopy colimits of diagrams of spaces can be used to construct toric varieties \cite{WZZ99}, one might hope to extend the work of Huh and Katz to non-representable matroids.

\subsection{Matroid Bundles}

Gelfand and MacPherson used oriented matroids to provide nice combinatorial formulas for characteristic classes of topological spaces with regular cell structures \cite{GM92} by associating an oriented matroid bundle to each real vector bundle. The Topological Representation Theorem for oriented matroids then allows one to construct a ``spherical'' bundle for each oriented matroid bundle.  Anderson and Davis showed that both of the above processes are functorial and their composition behaves like the forgetful functor where the zero section is deleted.  Thus, little information is lost in the \emph{combinatorialization} of real vector bundles \cite{Anderson99,AD99}.  It is natural to ask if we can combinatorialize vector bundles over unordered fields, such as $\C$.  Our hope is that the results in this paper will provide the toolbox for extending this line of work.

\begin{acknowledgements}
The author wishes to thank Eric Babson and Alexander Engstr\"{o}m for their many helpful conversations and suggestions.  This work was partially supported by the University of California, Davis Department of Mathematics, a U.S. Department of Education Graduate Assistance in Areas of National Need Fellowship, and National Science Foundation Grant \# DMS-0636297.\end{acknowledgements}

\appendix

\section{Appendix} 

Here we give an informal exposition on homotopy colimits.  This is merely intended to help readers who are unfamiliar with diagrams of spaces build some intuition on what these objects are and how they are useful.  For a more formal and complete introduction, we refer to the reader to \cite{WZZ99} and \cite{ZZ93}, especially for anyone coming from a more combinatorial background.

Although the language of diagrams and colimits is not commonly used in discrete geometry, it can be applied in many settings throughout the field.  For instance, a geometric simplicial complex is the colimit of a diagram over its face poset, ordered by reverse inclusion, where the spaces are geometric realizations of the simplices and the maps are simply the inclusion maps.  We illustrate this in the following example:

\begin{example} \label{apex1}
Let $\Delta$ be the simplicial complex consisting of two triangles which are glued along one edge.  The complex consists of four 0-simplices, five 1-simplices, and two 2-simplices which are glued together as drawn in Figure 4.

\begin{figure}[!h]
\begin{center} 
\includegraphics[width = 0.45\textwidth]{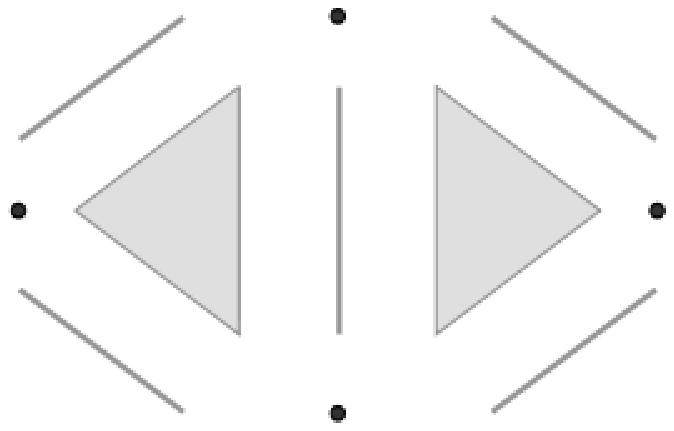}
\hspace{0.05\textwidth}
\includegraphics[width = 0.45\textwidth]{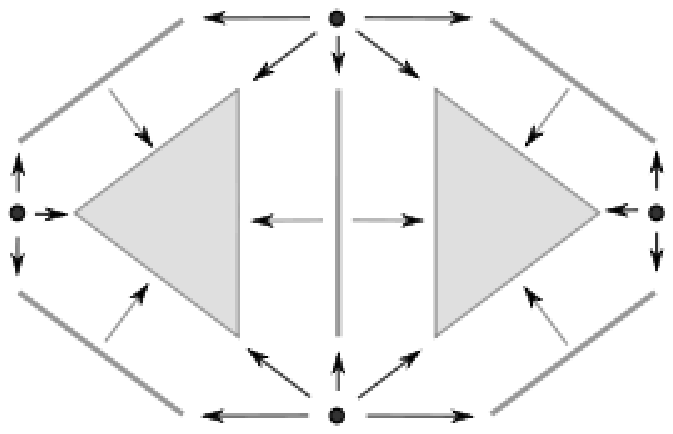}
\end{center}
\caption{The disjoint union of the simplices of $\Delta$ with indicated gluing maps}
\end{figure}

We can form a diagram of spaces $\D(F(\Delta))$ over the face poset $F(\Delta)$ of $\Delta$ by using the faces themselves as the spaces and the corresponding inclusion map for each relation.  The colimit of the diagram identifies the simplices along each of the gluing images yielding a geometric realization of $\Delta$. 

\begin{figure}[!h]
\begin{center} 
\includegraphics[width = 0.4\textwidth]{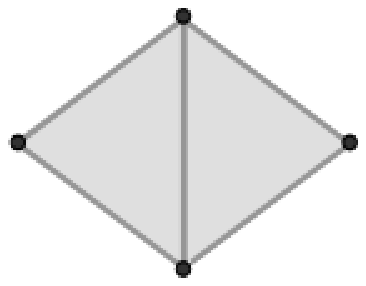}
\hspace{0.05\textwidth}
\includegraphics[width = 0.45\textwidth]{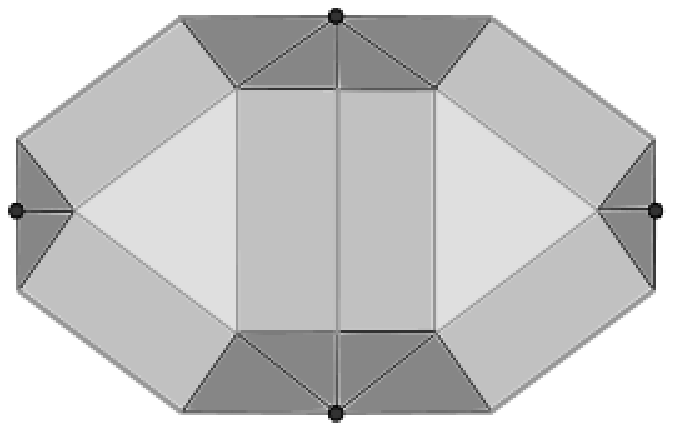}
\end{center}
\caption{The colimit (left) and homotopy colimit (right) of the diagram $\D(F(\Delta))$}
\end{figure}

The homotopy colimit of the diagram consists of the disjoint union of the simplices in $\Delta$ glued together with the mapping cylinders of each of the inclusions.  Each of the medium grey cells in Figure 5 is obtained from mapping a 1-simplex into a 2-simplex and each of the dark grey cells corresponds to the choice of mapping a 0-simplex either to a 1-simplex first and then to a 2-simplex or directly to a 2-simplex . 
\end{example}

Notice that the colimit and homotopy colimit of the diagram in Example \ref{apex1} have the same homotopy type.  This is not always the case, even for diagrams with well behaved spaces and maps.  In fact, a \emph{natural homotopy equivalence} of two $P$-diagrams, i.e. a morphism of diagrams $(id,\alpha)$ where $\alpha_p$ is a homotopy equivalence for each $p \in P$, cannot guarantee a  homotopy equivalence between colimits.  This is a significant drawback, even on an intuitive level, because if we take two collections of spaces which are componentwise homotopy equivalent and glue them together via the same combinatorial data, we want the resulting spaces to be homotopy equivalent as well.  The reason this does not happen for colimits is that one can make too many identifications and kill off interesting topology.  One should think of the homotopy colimit as gluing the spaces in a diagram together more carefully.  Consider the following example:
\begin{example} Let $P$ be the poset in Figure 6.
\begin{figure}[!h]
\begin{center} 
\includegraphics[width = 0.23\textwidth]{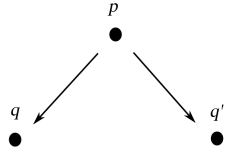}
\end{center}
\caption{A simple poset $P$}
\end{figure}

Define $P$-diagrams $\D$ and $\E$ by $D(p) = E(p) = S^1$, $D(q) = E(q) = E(q') = \bullet$, and $D(q') = D^2$ where every map is constant except for $d_{pq'}$, which is the inclusion map of $D(p)$ into the boundary of $D(q')$.

\begin{figure}[!h]
\begin{center} 
\includegraphics[width = 0.23\textwidth]{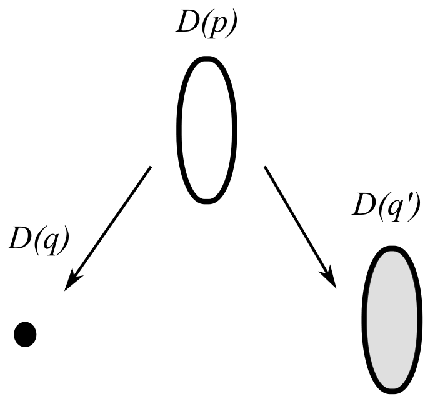}
\hspace{0.1\textwidth}
\includegraphics[width = 0.23\textwidth]{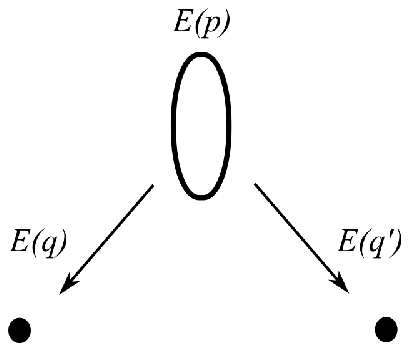}
\end{center}
\caption{The diagrams of spaces $\D$ (left) and $\E$ (right)}
\end{figure}

The diagrams $\D$ and $\E$ are naturally homotopy equivalent, yet their colimits, clearly have different homotopy types. In the colimit of $\D$, $D(p)$ is identified with both the boundary of $D(q')$ and the point $D(q)$ and hence, $\colim \D \simeq S^2$.  In the colimit of $\E$, all of the spaces are identified with a single point.

\begin{figure}[!h]
\begin{center} 
\includegraphics[width = 0.23\textwidth]{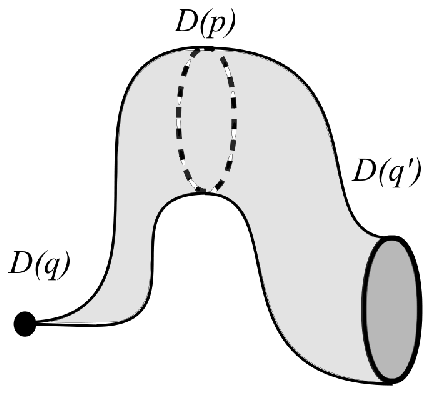}
\hspace{0.1\textwidth}
\includegraphics[width = 0.23\textwidth]{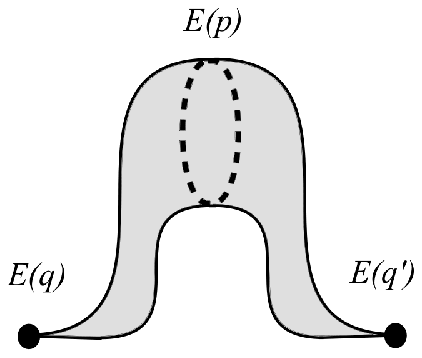}
\end{center}
\caption{The homotopy colimits of $\D$ and $\E$}
\end{figure}

Observe, however, that this problem is fixed if we glue in the mapping cylinder from each map in the diagram rather than simply making identifications.  In Figure 8, it is clear that the homotopy colimits of the diagrams are both homotopy equivalent to $S^2$.  Lemma \ref{HL} asserts that this will always be the case.
\end{example}

\bibliographystyle{plain}
\bibliography{references}{}

\end{document}